\documentclass[11pt,oneside,english]{amsart}
\usepackage[T1]{fontenc}
\usepackage[latin9]{inputenc}
\usepackage{geometry}
\usepackage{amsthm}
\usepackage{amssymb}
\usepackage{esint}

\makeatletter
\numberwithin{equation}{section}
\numberwithin{figure}{section}
\theoremstyle{plain}
\newtheorem{thm}{\protect\theoremname}[section]
  \theoremstyle{plain}
  \newtheorem{cor}[thm]{\protect\corollaryname}
  \theoremstyle{plain}
  \newtheorem{lem}[thm]{\protect\lemmaname}
  \theoremstyle{plain}
  \newtheorem{prop}[thm]{\protect\propositionname}
  \theoremstyle{remark}
  \newtheorem{rem}[thm]{\protect\remarkname}

\usepackage{amsthm}

\usepackage{amssymb}

\usepackage{amssymb}

\usepackage{amssymb}

\def\makebbb#1{
    \expandafter\gdef\csname#1\endcsname{
        \ensuremath{\Bbb{#1}}}
}
\makebbb{R}
\makebbb{N}
\makebbb{Z}
\makebbb{C}
\makebbb{D}
\makebbb{E}
\makebbb{H}
\makebbb{P}
\makebbb{B}
\makebbb{K}
\makebbb{E}

\usepackage{babel}

\usepackage{babel}

\usepackage{babel}

\usepackage{babel}

\makeatother

\usepackage{babel}
  \providecommand{\corollaryname}{Corollary}
  \providecommand{\lemmaname}{Lemma}
  \providecommand{\propositionname}{Proposition}
  \providecommand{\remarkname}{Remark}
\providecommand{\theoremname}{Theorem}

\begin{document}

\title{Tunneling, the Quillen metric and analytic torsion for high powers
of a holomorphic line bundle }
\begin{abstract}
Let $L$ be a line bundle over a compact complex manifold $X$ and
denote by $h_{L}$ and $h_{X}$ fixed Hermitian metrics on $L$ and
$TX,$ respectively. Our main result provides a formula for the average
distribution of the exponentially small eigenvalues of the corresponding
Dolbeault Laplacians associated to high tensor powers of $L,$ which
in physics terminology is a measure of ``tunneling'' of the Dolbeault
complex. Along the way a new proof of the asymptotics of the induced
Quillen metric on the corresponding determinant line is obtained.
A brief comparison with the tunneling effect for Witten Laplacians
and large deviation principles for fermions is also made.
\end{abstract}

\author{Robert J. Berman}

\email{robertb@chalmers.se}

\curraddr{Mathematical Sciences - Chalmers University of Technology and University
of Gothenburg - SE-412 96 Gothenburg, Sweden }


\maketitle
\tableofcontents{}

\section{Introduction}

\subsection{Setup}

Let $L\rightarrow X$ be a holomorphic line bundle over a compact
complex manifold $X$ and let $h_{L}$ be a smooth metric on the line
bundle $L$ with normalized curvature form $\omega$ (the most interesting
case will be when $\omega$ is not semi-positive). The normalization
is made so that $\omega$ is real and defines an integer cohomology
class: $[\omega]\in H^{2}(X,\Z).$ It will be convenient to use the
weight notation for $h_{L},$ i.e. fixing a local trivialization $s$
of $L$ we may locally write $\left\Vert s\right\Vert _{h_{L}}^{2}:=e^{-\phi}$
so that $\omega=dd^{c}\phi:=\frac{i}{2\pi}\partial\bar{\partial}\phi$
where $\phi$ will be called a \emph{weight} on $L.$ We also fix
a metric $h_{X}$ on $X$ and denote its volume form by $dV.$ These
metrics induce, in the standard way, Hermitian products on the space
$\Omega^{0,q}(X,L)$ of smooth $(0,q)-$forms with values in $L$
and we will write 
\[
\Delta_{\bar{\partial}}^{0,q}:=\bar{\partial}\bar{\partial}^{*}+\bar{\partial}^{*}\bar{\partial},
\]
for the corresponding Dolbeault-Kodaira Laplacian acting on $\Omega^{0,q}(X,L).$

We will be concerned with the asymptotic situation when $L$ is replaced
by a large tensor power, written in our additive notation as $kL.$
As is well-known, in terms of semi-classical analysis, $1/k$ plays
the role of Planck's constant and the main motivation for the present
paper is to understand the notion of ``tunneling'' for the Dolbeault
complex associated to the line bundle $kL$ equipped with a non-positively
curved metric, i.e. the distribution of the exponentially small eigenvalues
of the corresponding Dolbeault-Kodaira Laplacians $\Delta_{\bar{\partial}}^{0,q}.$
To be more precise, a sequence $\lambda_{k}$ of positive numbers
will be said to be \emph{exponentially small} if 
\[
\liminf_{k\rightarrow\infty}(\log\lambda_{k})/k<0
\]
(this is sometimes written as $\lambda_{k}=\mathcal{O}(e^{-\epsilon k})$
in the literature). As recalled in Section \ref{sub:Comparison-with-the}
the corresponding situation for the De Rham complex has been studied
extensively, motivated by the seminal work of Witten relating the
corresponding tunneling to supersymmetric quantum mechanics and Morse
theory and very precise results have been obtained. However, the situation
for the (asymptotic) Dolbeault complex appears to be somewhat mysterious
and there seem to be essentially no prior results available concerning
the corresponding tunneling. On the other hand, as shown by Demailly
\cite{d1}, who was inspired by Witten's approach to Morse theory,
the local weight $\phi$ of the given metric on $L$ plays the role
of a Morse function and this leads to asymptotic estimates on the
corresponding Dolbeault cohomology groups (the so called weak and
strong holomorphic Morse inequalities of Demailly). The main analytical
ingredient in Demailly's work is a formula for the leading asymptotics
of the number of lower-lying eigenvalues of the Dolbeault Laplacian
(see section \ref{sub:Comparison-with-the} for related results).
However, the formula in question gives no information on the distribution
of the corresponding exponentially small eigenvalues $\lambda_{k}$
or more precisely the distribution of the corresponding ``phases''
$(\log\lambda_{k})/k.$ One of the main points of the present paper
is to show that the Bergman kernel type asymptotics obtained in \cite{berm0}
can be used to give an explicit asymptotic formula for a certain average
of $(\log\lambda_{k})/k$ (see the concluding Section \ref{sub:Discussion-and-outlook}
for a discussion about much more precise conjectural results). \ref{sub:Discussion-and-outlook}
To this end a truncated version of the\emph{ }analytic torsion of
the Dolbeault complex will be introduced. Recall that the analytic
torsion,\emph{ }introduced by Ray-Singer,\emph{ }is defined as 
\[
T(h_{L},h_{X}):=\prod_{q=1}^{n}(\det\Delta_{\bar{\partial}}^{0,q})^{(-1)^{q+1}q},
\]
 using zeta function regularization, i.e. $\log\det\Delta_{\bar{\partial}}^{0,1}:=-\frac{\partial\zeta^{0,q}}{\partial s}_{s=0},$
where $\zeta^{0,q}(s)=\sum_{i}(\lambda_{i}^{0,q})^{-s}$ is the meromorphic
continuation to $\C$ of the zeta function for the \emph{non-zero
}eigenvalues $\{\lambda_{i}^{0,q}\}$ of $\Delta_{\bar{\partial}}^{0,q}$
(see \cite{b-g-s,m-m,so,bfl} and references therein). In the particular
case of a Riemann surface 
\[
T(h_{L},h_{X}):=\det\Delta_{\bar{\partial}}^{0,1}=\det\Delta_{\bar{\partial}}^{0,0}.
\]
The ``truncated'' version $T_{]0,\lambda[}(h,\omega_{0})$ of the
analytic torsion $T(h_{L},h_{X})$ is obtained by simply replacing
the regularized determinants with the product of all positive eigenvalues
strictly smaller than a given positive number $\lambda.$ Using functional
calculus we may hence write 
\begin{equation}
\log T_{]0,\lambda[}(h,\omega_{0}):=-\sum_{q=1}^{n}q(-1)^{q}\log\det(1_{]0,\lambda[}(\Delta_{\bar{\partial}}^{0,q}).\label{eq:def of truncated torsion}
\end{equation}
The asymptotics of the corresponding analytic torsions was studied,
using heat kernel asymptotics, in the seminal work of Bismut-Vasserot
\cite{b-v-} and is closely related to the Quillen metric on the determinant
of cohomology whose definition we briefly recall. Given a holomorphic
line bundle $L$ over $X$ the corresponding \emph{determinant line
$\mbox{DET}(L)$} is the one-dimensional complex vector space defined
as the following tensor product of $\bar{\partial}-$cohomology groups:
\[
\mbox{DET}(L):=\bigotimes_{q=1}^{n}\det(H^{q}(X,L))^{(-1)^{q}},
\]
 where $\det(W)$ denotes the top exterior power of a complex vector
space $W$. Given Hermitian metrics $h_{L}(=e^{-\phi_{L}})$ and $h_{X}$
on $L$ and $X,$ respectively, the corresponding \emph{Quillen metric}
on $\mbox{DET}(kL),$ that we will denote by $\mathcal{Q}(h_{L}^{\otimes k},h_{X}),$
is defined as the $L^{2}-$metric multiplied by the analytic torsion,
where the $L^{2}-$metric is the one induced from the isomorphisms
between the vector spaces $H^{q}(X,L)$ and the kernel of $\Delta_{\bar{\partial}}^{0,q}.$
Similarly, we will denote by $\mathcal{Q}_{]0,\lambda_{k}[}(h_{L}^{\otimes k},h_{X}),$
the ``truncated'' Quillen metrics obtained by replacing the analytic
torsion with its truncation, defined above.

Before stating our main results we also recall the following well-known
bi-functional on the space of metrics (weights) on $L:$

\[
\mathcal{E}(\phi,\phi'):=\frac{1}{(n+1)!}\int_{X}\sum_{j=0}^{n}(\phi-\phi')(dd^{c}\phi)^{n-j}\wedge(dd^{c}\phi')^{j}\left(=\tilde{ch}(h,h'))\right),
\]
where $\tilde{ch}(h,h')$ is (up to a multiplicative constant) the
Bott-Chern class attached to the first Chern class \cite{b-g-s,so}.
Moreover, in the case when $L$ is an ample line bundle and $\phi$
is a given smooth weight on $L$ we will write $P\phi$ for the semi-positively
curved weight defined as the following upper envelope: 
\[
P\phi=\sup\left\{ \psi:\,\psi\leq\phi\right\} ,
\]
where $\psi$ ranges over all smooth weights on $L$ with positive
curvature form. Then the functional $\mathcal{E}(P\phi,\phi)$ is
still well-defined by basic pluripotential theory, only using that
$P\phi$ is semi-positively curved and locally bounded (see \cite{b-b}
and references therein). Alternatively, by the regularity result in
\cite{berm1} $P\phi$ is locally $\mathcal{C}^{1.1}-$smooth and
hence the current $dd^{c}(P\phi)$ has locally bounded coefficients
and exists point-wise almost everywhere on $X.$ As a consequence
$\mathcal{E}(P\phi,\phi)$ can be defined as a standard Lebesgue integral
over $X.$

\subsection{Statement of the main results}
\begin{thm}
\label{thm:analyt torsion for ample}Let $L\rightarrow X$ be an ample
line bundle equipped with a smooth Hermitian metric $h_{L}$ and let
$h_{X}$ be a fixed smooth Hermitian metric on $X$ (which is not
assumed to be Kähler). Let $\lambda_{k}$ be a sequence such that
$\lambda_{k}=o(k)$ and such that $\lambda_{k}$ is not exponentially
small. Then the large $k-$limits of the corresponding analytic torsions
$\frac{1}{k^{n+1}}\log T(h_{L}^{\otimes k},h_{X})$ and their ``truncations''
$\frac{1}{k^{n+1}}\log T_{]0,\lambda_{k}[}(h_{L}^{\otimes k},h_{X})$
exist and both coincide with $\mathcal{E}(P\phi_{L},\phi_{L})$
\end{thm}
In fact, it will be clear that the only contribution to the truncated
analytic torsions in the previous theorem comes from the exponentially
small eigenvalues (in an appropriate sense). In the case when $h_{X}$
is a Kähler metric the asymptotics of the (non-truncated) analytic
torsion in the previous theorem were deduced in \cite{b-b} from the
exact anomaly formula in \cite{b-g-s} for the Quillen metric (referred
to above) combined with the asymptotics for the $L^{2}-$part of the
Quillen metric proved in \cite{b-b}. Of course, since $L$ is ample
the manifold $X$ is automatically Kähler, but the point is that one
does not need to assume that the Hermitian metric $h_{X}$ is Kähler.

As a consequence of the previous theorem we get the following geometric
criterion for the existence of exponentially small eigenvalues of
$\bar{\partial}-$Laplacians:
\begin{cor}
\label{cor:exist of small eig}Let $L$ be a line bundle and $h_{L}$
and $h_{X}$ metrics as in the previous theorem. Suppose that 
\[
\mathcal{E}(P\phi_{L},\phi_{L})\neq0.
\]
 Then, for some $q,$ there is a sequence of exponentially small eigenvalues
of the corresponding $\bar{\partial}-$Laplacian acting on $(0,q)-$forms
with values in $kL.$ In the particular case of a line bundle of positive
degree over a Riemann surface equipped with a metric $h_{L}$ whose
curvature is not semi-positive everywhere, the condition above is
always satisfied. More precisely, 
\begin{equation}
\frac{1}{k^{2}}\log\sum_{i}\lambda_{i}^{(k)}\rightarrow-\frac{1}{2}\left\Vert d(P\phi-\phi)\right\Vert _{X}^{2}(<0),\label{eq:asym in them exp small}
\end{equation}
 where $(\lambda_{i}^{(k)})$ are the ``small'' positive eigenvalues
of the $\bar{\partial}-$Laplacian on the space of smooth sections
with values in $kL$ (i.e the eigenvalues are in $]0,k^{1-\epsilon}[)$
\end{cor}
One of the main ingredients in the proof of the previous theorem is
an asymptotic anomaly formula for a truncated version of the Quillen
metric, which leads to the following 
\begin{thm}
\label{thm:quillen}Let $L\rightarrow X$ be a holomorphic line bundle
over a compact complex manifold (not necessarily Kähler). Fix a Hermitian
metric $h_{X}$ on $X$ and two metrics $h_{L}$ and $h_{L}'$ on
$L.$ Then the following asymptotic anomaly formula for the corresponding
Quillen metrics on the determinant lines $\mbox{DET}(kL),$ as well
as its truncations $\mathcal{Q}_{]0,\lambda_{k}[}(h_{L}^{\otimes k},h_{X}),$
holds 
\[
\lim_{k\rightarrow\infty}\frac{1}{k^{n+1}}\mathcal{\log}(\frac{\mathcal{Q}(h_{L}^{\otimes k},h_{X})}{\mathcal{Q}(h_{L'}^{\otimes k},h_{X}))})=\tilde{ch}(h_{L},h_{L}'),
\]
for any sequence $\lambda_{k}$ as in the previous theorem.
\end{thm}
In the case when $h_{X}$ is a Kähler metric the previous theorem
is a direct consequence of the deep \emph{exact} anomaly formula for
the Quillen metric of Bismut-Gillet-Soulé for the determinant line
applied to any fixed tensor power of $L$ (see \cite{b-g-s}, Theorem
0.3 stated in part I). Even though the exact formula is not known
in the non-Kähler case it should be stressed that the previous theorem
can be deduced from the leading heat-kernel asymptotics in \cite{bi,b-vas,m-m},
as they are independent of any Kähler assumptions (see Remark \ref{rem:referee}).
As explained below the key point of the present proof, where the role
of the heat-kernel is played by Bergman type kernels \cite{berm0},
is to first obtain the asymptotics for the ``truncated'' Quillen
metric obtained by replacing the analytic torsions with suitable truncations.

By standard arguments the formula in the previous theorem can be used
to obtain an asymptotic version of Theorem 0.1 in \cite{b-g-s} in
the case when $\pi:\, X\rightarrow S$ is a submersion of a base $S$
and a pair $(h_{L},h_{X}$) as above is fixed. Then the curvature
of the Quillen metric on $\mbox{DET}(kL),$ seen as a line bundle
over $S,$ converges when divided by $k^{n+1}$ to the push-forward
to $S$ of the top exterior power of the curvature form of $h_{L}.$ 

The starting point of the proof of Theorem \ref{thm:quillen} is the
observation that differential of the normalized logarithm $\mathcal{Q}_{]0,\lambda_{k}[}$
of the ``truncated'' Quillen metric on $\mbox{DET}(kL),$ seen as
a functional on the space of all Hermitian metrics $\phi$ (or rather
weights) on $L,$ is represented by the alternating sum of the one-point
measures (i.e. Bergman type measures) associated to the spaces of
``low-energy'' $(0,q)-$forms. Using the large $k$ asymptotics
of the latter measures obtained in \cite{berm0} gives 
\[
\frac{1}{k^{n+1}}d_{\phi}\mathcal{Q}_{]0,\lambda_{k}[}(k\phi,h_{X})\rightarrow(dd^{c}\phi)^{n}/n!
\]
and integrating between the line segment connecting $\phi_{0}$ and
$\phi_{1}$ in the space of all weights on $L$ then concludes the
proof of Theorem \ref{thm:quillen} for the ``truncated'' Quillen
metrics. Adapting arguments in \cite{b-vas} to the present setting
we also show that asymptotically the analytic torsions may be replaced
by truncations as above, concluding the proof of Theorem \ref{thm:quillen}.
Theorem \ref{thm:analyt torsion for ample} is then obtained by using
the asymptotics for the $H^{0}(X,kL)-$part of the Quillen metrics
for a general smooth (but possible non-positively curved) metric in
\cite{b-b} (or alternatively the asymptotics for the corresponding
one-point (Bergman) measures in \cite{berm1}).

\subsection*{Acknowledgments}

The author is grateful to Johannes Sjöstrand and Frédéric Faure for
their interest in this work and for stimulating discussions related
to the topic of the present paper, as well as to the referees whose
comments helped to improve the exposition of the paper. This work
has been supported by grants from the Swedish and European Research
Councils and the Wallenberg Foundation.

\section{Proofs of the main results}

In the following we will fix the metric $h_{X}$ on $TX.$ Then any
given weight $\phi$ (which we will eventually vary) induces Hermitian
products $\left\langle \cdot,\cdot\right\rangle _{\phi}$ on the space
$\Omega^{0,q}(X,L)$ using the metrics $(e^{-\phi},h_{X})$ on $L$
and $TX,$ respectively. Given $\phi$ we denote the associated $\bar{\partial}-$Laplacians
on $\Omega^{0,q}(X,L)$ by $\Delta^{0,q}$ (and sometimes by $\Delta_{\phi}^{0,q}$
to indicate the dependence on $\phi).$ We will denote by $\mathcal{H}_{]0,\lambda[}^{0,q}$
the subspace of $\Omega^{0,q}(X,L)$ spanned by the eigenforms of
$\Delta^{0,q}$ with eigenvalues in $]0,\lambda[$ (sometimes called
the\emph{ space} of \emph{low-energy $(0,q)-$forms}). To this space
we associate its (non-normalized) one-point measure, i.e. 
\begin{equation}
\B_{]0,\lambda[}^{0,q}:=\sum_{i}|\Psi_{i}^{0,q}|_{\phi}^{2}dV,\label{eq:def of one correl measure with eigenv}
\end{equation}
 where $(\Psi_{i}^{0,q})$ is any orthonormal base for $\mathcal{H}_{]0,\lambda[}^{0,q}.$
We will also use the notation $\mathcal{H}^{0,q}:=\ker\Delta^{0,q}$
and write $\B^{0,q}$ for the corresponding one-point measure.

\subsection{Variational formulae}

If $\mathcal{F}(\phi)$ is a (Gateaux differentiable) functional on
the affine space of all weights on $L$ we will write $d\mathcal{F}$
for its differential, which is a one-form on the space of all weights.
We will identify the linear functional $d\mathcal{F}_{|\phi}$ on
$\mathcal{C}^{\infty}(X)$ with a measure in the usual way. Concretely,
this means that 
\begin{equation}
\frac{d}{dt}\mathcal{F}(\phi_{t})=\int_{X}(d\mathcal{F}_{|\phi_{t}})\frac{d\phi_{t}}{dt}.\label{eq:def of var der}
\end{equation}
Given a pair of weights $\phi$ and $\phi'$ on $L$ we now let 
\[
\mathcal{L}^{0,q}(\phi,\phi'):=-\log\det(\left\langle \Psi_{i}^{0,q},\Psi_{j}^{0,q}\right\rangle _{\phi})
\]
 where $(\Psi_{i}^{0,q})_{i}$ is a basis in $\ker\Delta_{\phi'}^{0,q}$
which is orthonormal wrt $\left\langle \cdot,\cdot\right\rangle _{\phi_{'}}.$
We then have the following basic
\begin{lem}
\label{lem:deriv of e and l}Fix a weight $\phi'$ on $L$ and write
$\mathcal{L}^{0,q}(\phi):=\mathcal{L}^{0,q}(\phi,\phi')$ and $\mathcal{E}(\phi):=\mathcal{E}(\phi,\phi')$.
Then 
\[
(i)\, d\mathcal{E}_{|\phi}=\frac{(dd^{c}\phi)^{n}}{n!},\,\,\,\,(ii)\, d\mathcal{L}^{0,q}{}_{|\phi}=\B^{0,q}
\]
 \end{lem}
\begin{proof}
The first point was first shown by Mabuchi \cite[Theorem 2.3]{ma}
(see also \cite{b-b} and references therein for a more general setting).
As for the second point it holds in a general setting where $\Pi_{\phi}$
is the orthogonal projection (wrt the Hermitian product $\left\langle \cdot,\cdot\right\rangle _{\phi})$
on a given $N-$dimensional subspace of $\mathcal{C}^{\infty}(X,L\otimes E)$
where $E$ is a given Hermitian complex vector bundle (here $E=\Lambda^{0,q}(X,h_{X})).$
Indeed, we can then compute the time derivative of the corresponding
functional $\mathcal{L}(\phi_{t},\phi)$ (here $\mathcal{L}^{0,q}(\phi,\phi'))$
using the basic formula

\begin{equation}
\frac{d}{dt}_{t=0}\log\det(H(t))=\mbox{Tr}(H(0)^{-1}\frac{d}{dt}_{t=0}H(t))\label{eq:der of log as trac}
\end{equation}
By the cocycle property $\mathcal{L}(\phi_{0},\phi_{1})+\mathcal{L}(\phi_{1},\phi_{2})+\mathcal{L}(\phi_{2},\phi_{0})=0$
we have that $\frac{d}{dt}\mathcal{L}(\phi_{t},\phi_{'})$ is independent
of $\phi'$ and hence we can set $\phi'=\phi_{0}$ above so that $H(0)=I,$
which gives the desired formula. 
\end{proof}
We will also have great use for the following lemma, whose proof is
inspired by some arguments in \cite{so}:
\begin{lem}
\label{lem:deriv of trunc analyt torsion}The differential of the
functional $\phi\mapsto\tau_{\lambda}(\phi):=\log T_{]0,\lambda[}(e^{-\phi},h_{X})$
given by the truncated analytic torsion (formula \ref{eq:def of truncated torsion})
is given by 
\[
d\tau_{|\phi}=\sum_{q=1}^{n}(-1)^{q}\B_{]0,\lambda[}^{0,q}
\]
for a generic number $\lambda.$ In particular, the differential of
the truncated Quillen metric wrt $\phi$ satisfies 
\begin{equation}
d_{\phi}\mathcal{Q}_{]0,\lambda_{k}[}(e^{-\phi},h_{X})=\sum_{q=1}^{n}(-1)^{q}\B_{[0,\lambda[}^{0,q}.\label{eq:differential of Quillen metric}
\end{equation}
\end{lem}
\begin{proof}
In the proof we will repeatedly use that $\bar{\partial}$ commutes
with $\Delta(:=\Delta_{\bar{\partial}})$ and hence if $\Psi^{0,q}$
is an eigenform of $\Delta^{0,q}$ then $\bar{\partial}\Psi^{0,q}$
is an eigenform of $\Delta^{0,q+1}$ unless it vanishes identically.
To fix ideas we start with

\emph{The case $n=1:$}

Since, $\Delta=\bar{\partial}^{*}\bar{\partial}$ on $\Omega^{0,q}(X,L)$
it follows immediately from the definition of the determinant that
\[
\det(1_{]0,\lambda[}(\Delta_{\bar{\partial}}^{(0)})=\det(\left\langle \Psi_{i},\Psi_{j}\right\rangle )^{-1}\det(\left\langle \bar{\partial}\Psi_{i},\bar{\partial}\Psi_{j}\right\rangle )
\]
 for any given base in $\mathcal{H}_{]0,\lambda[}^{0,q}$ (with $q=0).$
Next we note that, given the path $\phi_{t},$ we may find a path
$(\Psi_{i}^{0,q})(t)$ of bases in $\mathcal{H}_{]0,\lambda[}^{0,q}(t)$
such that $(\Psi_{i}^{0,q})(t)$ is orthonormal wrt $\phi_{0}$ for
$t=0$ and for any fixed index $i$ $(\Psi_{i}^{0,q})(0)$ is an eigenform
for $\Delta_{\phi_{0}}$ and 
\begin{equation}
\left\langle \frac{d}{dt}_{t=0}(\Psi_{i}^{0,q})(t),\mathcal{H}_{]0,\lambda[}^{0,q}(0)\right\rangle _{\phi_{0}}=0.\label{eq:og gauge}
\end{equation}
 To see this we first recall that, as a well-known consequence of
the ellipticity of Laplacians, for a generic $\lambda$ the family
$\mathcal{H}_{[0,\lambda[}^{0,q}(t)$ defines a vector bundle over
$\{t\}:=]-\epsilon,\epsilon[$ (see Proposition  1 on p. 123 in \cite{so}).
Moreover, since, for any $t,$ $\ker\Delta_{\bar{\partial}}^{0,q}(t)$
is isomorphic to the Dolbeault cohomology group $H_{\bar{\partial}}^{0,q}(X,L)$
it follows that $F:=\mathcal{H}_{]0,\lambda[}^{0,q}(t)$ is also a
vector bundle. Hence, we can start with an arbitrary smooth curve
$(s_{i})(t)$ of bases in $F_{t}$ satisfying the first requirements
above at $t=0.$ Then, for $t$ sufficiently small, $(\Psi_{i}^{0,q})(t):=s_{i}(t)-t\Pi_{t}(\frac{d}{dt}_{t=0}(s_{i})(t))$
has the desired properties, since $\Pi_{0}(\frac{d}{dt}_{t=0}(\Psi_{i}^{0,q})(t))=0,$
where $\Pi_{t}$ denotes the orthogonal projection onto $\mathcal{H}_{]0,\lambda[}^{0,q}(t).$

Now we can decompose 
\begin{equation}
\frac{d}{dt}_{t=0}\log\det(1_{]0,\lambda[}(\Delta_{\bar{\partial}}^{(0)})=-\frac{d}{dt}_{t=0}\log\det(\left\langle \Psi_{i}(t),\Psi_{j}(t)\right\rangle +\frac{d}{dt}_{t=0}\log\det(\left\langle \bar{\partial}\Psi_{i}(t),\bar{\partial}\Psi_{j}(t)\right\rangle ).\label{eq:pf of deriv of logdet lapla}
\end{equation}
Using formula formula \ref{eq:der of log as trac} and the fact that
\[
\frac{d}{dt}_{t=0}(\left\langle \Psi_{i}(t),\Psi_{j}(t)\right\rangle _{\phi_{t}}=\left\langle -\frac{d\phi_{t}}{dt}\Psi_{i}(t),\Psi_{j}(t)\right\rangle +0
\]
(by Leibniz rule and \ref{eq:og gauge}) the first term in \ref{eq:pf of deriv of logdet lapla}
above gives (also using that $H(0)=I)$ 
\[
-\frac{d}{dt}_{t=0}\log\det(\left\langle \Psi_{i}(t),\Psi_{j}(t)\right\rangle =\int_{X}\B_{]0,\lambda[}^{0,0}\frac{d\phi_{t}}{dt}.
\]
The second term is computed similarly, using that $\alpha_{i}:=\bar{\partial}\Psi_{i}/\left\Vert \bar{\partial}\Psi_{i}\right\Vert _{\phi_{0}}$
is an orthonormal base in $\mathcal{H}_{]0,\lambda[}^{0,1}(t)$ for
$t=0$ and that the relation \ref{eq:og gauge} still holds with $\Psi_{i}$
replaced by $\alpha_{i}$ ( indeed, since $\bar{\partial}$ commutes
with $\frac{d}{dt}$ and $\Delta_{\bar{\partial}}$ we get $\Pi_{]0,\lambda[}^{0,q+1}(\frac{d}{dt}\bar{\partial}\Psi)=\bar{\partial}\Pi_{]0,\lambda[}^{0,q}(\frac{d}{dt}\Psi)=0).$ 

\emph{General $n$}

By Hodge theory we have a decomposition 
\[
\Omega^{0}(X,L)=\ker\Delta_{\bar{\partial}}\oplus\mbox{Im}\bar{\partial}\oplus\mbox{}\bar{\partial}^{*},
\]
which is orthogonal wrt the corresponding Hermitian product. Restricting
to $\mathcal{H}_{]0,\lambda[}^{0,q}$ (which by definition is in the
orthogonal complement of $\ker\Delta_{\bar{\partial}}$ in $\Omega^{0}(X,L))$
gives an induced orthogonal decomposition 
\[
\mathcal{H}_{]0,\lambda[}^{0,q}=M_{q}\oplus N_{q}.
\]
Next, we note that $\bar{\partial}$ induces a bijection, intertwining
the corresponding restricted Laplacians, such that 
\begin{equation}
\bar{\partial}:\,\,\, N_{q}\rightarrow M_{q+1},\,\,\,\Delta_{N_{q}}=\bar{\partial}^{*}\bar{\partial}\label{eq:of of deriv of log det: bijec}
\end{equation}
 and hence $\det(\Delta_{M_{q}})=\det(\Delta_{N_{q-1}})$ so that
$\det(1_{]0,\lambda[}(\Delta_{\bar{\partial}}^{(q)})=\det(\Delta_{N_{q-1}})\det(\Delta_{N_{q}}).$
The latter relation implies, since $\det(\Delta_{N_{-1}}):=1=\det(\Delta_{N_{n}})$
that 
\begin{equation}
-\sum_{q=1}^{n}q(-1)^{q}\log\det(1_{]0,\lambda[}(\Delta_{\bar{\partial}}^{(q)})=\sum_{q=1}^{n}(-1)^{q}\log\det(\Delta_{N_{q}}).\label{eq:log det using hodge}
\end{equation}
Using the bijection \ref{eq:of of deriv of log det: bijec} we can
now repeat the arguments used above (when $n=1)$ to deduce that 
\[
\frac{d}{dt}_{t=0}\log(\det(\Delta_{N_{q}})=\int_{X}(\B_{]0,\lambda[\cap N_{q}}^{0,q}-\sum_{q=1}^{n}\B_{]0,\lambda[\cap M_{q+1}}^{0,q+1})\frac{d\phi_{t}}{dt},
\]
where the intersection with $N_{q}$ indicates that we have replaced
$\mathcal{H}_{]0,\lambda[}^{0,q}$ with its subspace $N_{q}$ in the
definition \ref{eq:def of one correl measure with eigenv} (and similarly
for $M_{q+1}).$ Hence, taking the alternating sum over $q$ and using
\ref{eq:log det using hodge} proves the lemma in general dimensions.
Finally, formula \ref{eq:differential of Quillen metric} now follows
by invoking the formula $(ii)$ in the previous lemma. 
\end{proof}

\subsection{Asymptotics}

Next we recall the following asymptotics from \cite{berm0} (which
can be seen as a local version of Demailly's strong holomorphic Morse
inequalities):
\begin{prop}
\label{pro:strong morse}Let $\lambda_{k}$ be a sequence of positive
numbers such that $\lambda_{k}=O(k^{1-\epsilon})$ for some $\epsilon<1/2.$
Then the sequence $k^{-n}\B_{[0,\lambda_{k}[}^{0,q}/dV$ is uniformly
bounded and \textup{
\[
k^{-n}\B_{[0,\lambda_{k}[}^{0,q}\rightarrow(-1)^{q}1_{X(q)}(dd^{c}\phi)^{n}/n!
\]
weakly as $k\rightarrow\infty,$ where $X(q)$ is the subset of $X$
where $dd^{c}\phi$ has exactly $q$ negative eigenvalues. In particular,
\[
\mathcal{\dim H}_{[0,\lambda_{k}[}^{0,q}=k^{n}\int_{X}(-1)^{q}1_{X(q)}(dd^{c}\phi)^{n}/n!+o(k^{n}).
\]
}\end{prop}
\begin{proof}
As shown in \cite{berm0} (Proposition  5.1) the upper bound in the
convergence above holds for any sequence $\lambda_{k}=k\mu_{k}$ with
$\mu_{k}\rightarrow0.$ As for the lower bound it was was shown to
hold as long as $\delta_{k}/\mu_{k}\rightarrow0$ for $\delta_{k}$
a certain non-explicit sequence tending to zero (Proposition  5.3).
This latter sequence appears in Lemma 5.2 in \cite{berm0} and the
proof given there actually shows that $\delta_{k}$ can be taken as
$\delta_{k}=1/k^{1/2-\delta}$ for any $\delta>0$ (since the radius
$R_{k}$ appearing in that proof is equal to $\log k).$ A further
refinement of this argument will be considered in Proposition  \ref{pro:quasi-mode}.
\end{proof}
The previous asymptotics will allow us to obtain the asymptotics of
truncated analytic torsions (and hence of exponentially small eigenvalues).
Using the next proposition we will then deduce the corresponding asymptotics
for the usual analytic torsions.
\begin{prop}
\label{pro:anal torsion and its trunc}The following asymptotics hold
for any line bundle $L\rightarrow X$ and smooth Hermitian metrics
on $L$ and $X:$ 
\[
\log(\det(\Delta_{k}^{(q)})=\log(\det(1_{]0,\lambda_{k}]}(\Delta_{k}^{(q)}))+O(k^{n+\epsilon})
\]
if $\lambda_{k}=O(k^{1-\epsilon})$ for a given $\epsilon\in]0,1[.$ \end{prop}
\begin{proof}
We will adapt to our setting some arguments from \cite{b-vas}, where
it was among other things shown that $\log(\det(\Delta_{k}^{(q)})=O(k^{n})\log k$
if the metric on $L$ is positively curved (the point being that the
smallest positive eigenvalue of $\Delta_{k}^{(q)}$ is then of the
order $k).$ As follows immediately from the definition (see below)
the statement of the proposition to be proved is equivalent to 
\begin{equation}
\zeta_{1_{[\lambda_{k},\infty[}(\Delta^{(q)})}'(0)=O(k^{n+\epsilon}).\label{eq:equiv form of prop anal trunc}
\end{equation}
To prove the latter asymptotics we first recall some general facts
about spectral zeta functions (following chapter V in \cite{so}).
If $D$ is an Hermitian non-negative operator on a Hilbert space then
its spectral zeta function is defined by 
\[
\zeta_{D}(s)=\frac{1}{\Gamma(s)}\int_{0}^{\infty}(t^{s}\mbox{(Tr(}e^{-tD})-\Pi_{D})\frac{dt}{t}
\]
where $\Pi_{D}$ denotes the orthogonal projection on the kernel of
$D$ (we assume that $\zeta_{D}(s)$ exists for $s$ a complex number
of sufficiently negative real part and is then analytically contained
to other values for $s$). In other words $\zeta_{D}(s)$ is the Mellin
transform of the spectral theta function 
\[
\Theta_{D}(t):=\mbox{(Tr(}e^{-tD})-\Pi_{D})=\sum_{\nu_{i}>0}e^{-t\nu_{i}}
\]
 summing over the positive eigenvalues of $D.$ Next we assume that 
\begin{itemize}
\item $\Theta_{D}(t)$ converges for $t>0$
\item For every positive integer $M$ there are real numbers $a_{i}=0$
with $a_{i}=0$ for $i<-n$ such that 
\begin{equation}
\Theta_{D}(t)=\sum_{j=-n}^{j=M}a_{j}t^{j}+O(t^{M+1})\label{eq:exp of theta}
\end{equation}
uniformly when $t\rightarrow0.$
\end{itemize}
Under these assumptions one obtains (see formula 11 on p. 99 in \cite{so})
\begin{equation}
\zeta{}_{D}(0)=a_{0},\,\,\,\,\,\zeta'_{D}(0)=\int_{1}^{\infty}\Theta_{D}(t)\frac{dt}{t}+\left(\gamma a_{0}+\sum_{j<0}\frac{a_{j}}{j}+\int_{0}^{1}\rho_{0}(t)\frac{dt}{t}\right)\label{eq:propert of mellin}
\end{equation}
where $\rho_{0}(t)$ is the $O(t^{M+1})$ term in \ref{eq:exp of theta}
for $M=0$ and $\gamma$ denotes Euler's constant. We also define
a scaled version of $\zeta_{D}$ by setting $\widetilde{\zeta}_{D}:=k^{-n}\zeta_{k^{-1}D}$
so that the derivative $\zeta'_{D}(0)$ at $s=0$ satisfies 
\begin{equation}
k^{-n}\zeta'_{D}(0)=-\log k\widetilde{\zeta}_{D}(0)+\widetilde{\zeta}'_{D}(0)\label{eq:chain rule for zeta}
\end{equation}
(compare formula 40 in \cite{b-vas}). 

We now come back to the present complex geometric setting. As shown
in \cite{b-vas} (Theorem 2), for any positive integer $M,$ there
is an asymptotic expansion 
\begin{equation}
k^{-n}\mbox{Tr(}e^{-tk^{-1}\Delta^{(q)}})\frac{dt}{t}=\sum_{j=-n}^{j=M}a_{j}^{q}(k)t^{j}+O(t^{M+1}),\label{eq:asympt exp of heat}
\end{equation}
when $t\rightarrow0$ uniformly for $t\in[0,1]$ and moreover 
\begin{equation}
a_{j}^{q}(k)=a_{j}^{q}+O(k^{-1/2}),\label{eq:conv of heat coeff}
\end{equation}
as $k\rightarrow\infty.$ Next, we note that, for $t\in[0,1]$ 
\begin{equation}
k^{-n}\mbox{Tr(}e^{-tk^{-1}\Delta^{(q)}})-k^{-n}\mbox{Tr(}e^{-tk^{-1}1_{[k^{1-\epsilon},\infty[}(\Delta^{(q)})})=\int(-1)^{q}1_{X(q)}(dd^{c}\phi)^{n}/n!+o(1)\label{eq:contrib of small eigenv}
\end{equation}
 uniformly when $k\rightarrow\infty$ (strictly speaking these asymptotics
only hold when $\epsilon<1/2$ but in general the argument will show
that the lhs above is uniformly bounded which is all that will be
used to deduce the lemma). Indeed, the left hand side above may be
written and estimated from below and above as 
\[
e^{-tk^{-\epsilon}}k^{-n}\mathcal{\dim H}_{]0,\lambda_{k}[}^{0,q}\leq k^{-n}\sum_{i}e^{-tk^{-1}\nu_{i,k}}\leq k^{-n}\mathcal{\dim H}_{]0,\lambda_{k}[}^{0,q},
\]
which combined with Proposition  \ref{pro:strong morse} proves the
previous formula. Now combining \ref{eq:asympt exp of heat}, \ref{eq:conv of heat coeff}
and \ref{eq:asympt exp of trunc heat} gives an expansion 
\begin{equation}
k^{-n}\mbox{Tr(}e^{-tk^{-1}1_{[k^{1-\epsilon},\infty[}(\Delta^{(q)})})=\sum_{j=-n}^{j=M}b_{j}^{q}t^{j}+O(t^{M+1})+o(1)\label{eq:asympt exp of trunc heat}
\end{equation}
uniformly for $t\in[0,1]$ as $k\rightarrow\infty$ (where $b_{j}^{q}=a_{j}^{q}$
for $j\neq0).$ Accordingly, it follows from \ref{eq:propert of mellin}
that 
\begin{equation}
\widetilde{\zeta}{}_{\Delta^{(q)}}(0)=-b_{0}^{q}(k)=-b_{0}^{q}+o(1)\label{eq:zeta function of trunc}
\end{equation}
and 
\[
\widetilde{\zeta}{}_{1_{[\lambda_{k},\infty[}(\Delta^{(q)})}(0)=-(b_{0}^{q}-\int_{X(q)}(-1)^{q}(dd^{c}\phi)^{n}/n!)+o(1).
\]
Next, we apply the derivative formula in \ref{eq:propert of mellin}
to $D=k^{-1}1_{[k^{1-\epsilon},\infty[}(\Delta^{(q)}$ to get 
\[
\widetilde{\zeta}_{1_{[k^{1-\epsilon},\infty[}(\Delta^{(q)})}'(0)=\int_{1}^{\infty}k^{-n}\mbox{Tr(}e^{-tk^{-1}1_{[k^{1-\epsilon},,\infty[}(\Delta^{(q)})})\frac{dt}{t}+O(1),
\]
also using \ref{eq:asympt exp of trunc heat}. From the basic estimate
\[
k^{-n}\mbox{Tr(}e^{-tk^{-1}1_{[\lambda_{k},\infty[}(\Delta^{(q)})})\leq e^{-(t-1)k^{-1}\lambda_{k}}(k^{-n}\mbox{Tr(}e^{-k^{-1}(\Delta^{(q)})}))
\]
we deduce, since the second factor is uniformly bounded according
to \ref{eq:asympt exp of heat} applied to $t=1$ and since we have
assumed $\lambda_{k}=O(k^{1-\epsilon})$ the bound 
\[
\int_{1}^{\infty}k^{-n}\mbox{Tr(}e^{-tk^{-1}1_{[k^{1-\epsilon},,\infty[}(\Delta^{(q)})})\frac{dt}{t}\leq C\int_{1}^{\infty}e^{-(t-1)k^{-\epsilon}}\frac{dt}{t}\leq C'k^{\epsilon}.
\]
All in all this means that $\widetilde{\zeta}_{1_{[k^{1-\epsilon},\infty[}(\Delta^{(q)})}'(0)=O(k^{\epsilon})$
which combined with \ref{eq:chain rule for zeta} and \ref{eq:zeta function of trunc}
proves \ref{eq:equiv form of prop anal trunc} and hence finishes
the proof of the proposition. 
\end{proof}

\subsection{Proof of Theorem \ref{thm:quillen}}

We start with a given $\phi$ and take a smooth path $\phi_{t}$ such
that $\phi_{1}=\phi$ and $\phi_{0}$ has positive curvature. Setting
\begin{equation}
f_{k}(t):=k^{-(n+1)}(\log T_{]0,\lambda_{k}[}(e^{-k\phi_{t}},h_{X})+\sum_{q=0}^{n}(-1)^{q}\mathcal{L}^{0,q}(k\phi_{t})\label{eq:pf of thm analyt}
\end{equation}
we have to prove that

\begin{equation}
\frac{1}{V}\lim_{k\rightarrow\infty}(f_{k}(1)-f_{k}(0))=\mathcal{E}(\phi_{1})-\mathcal{E}(\phi_{0})(:=\mathcal{E}(\phi_{1},\phi_{0})).\label{eq:step one in conv of anal torsion}
\end{equation}
To this end note that combining Lemma \ref{lem:deriv of e and l}
and Lemma \ref{lem:deriv of trunc analyt torsion} gives $\frac{df_{k}(t)}{dt}=$
\[
=k^{-n}\int_{X}(\sum_{q=1}^{n}(-1)^{q}(\B_{\phi_{t},]0,\lambda_{k}[}^{0,q}+\B_{\phi_{t}}^{0,1})\frac{d\phi_{t}}{dt}=k^{-n}\int_{X}(\sum_{q=1}^{n}(-1)^{q}\B_{\phi,[0,\lambda_{k}[}^{0,q})\frac{d\phi_{t}}{dt}.
\]
 Using Proposition  \ref{pro:strong morse} together with the dominated
convergence theorem hence gives 
\[
\frac{1}{V}\lim_{k\rightarrow\infty}(f_{k}(1)-f_{k}(0))=\lim_{k\rightarrow\infty}\int_{0}^{1}\frac{df_{k}(t)}{dt}dt=\frac{1}{Vn!}\int_{0}^{1}\int_{X}(dd^{c}\phi_{t})^{n}\frac{d\phi_{t}}{dt}dt.
\]
According to the variational characterization of $\mathcal{E}$ in
Lemma \ref{lem:deriv of e and l} (or by directly computing the integral)
this finishes the proof of \ref{eq:step one in conv of anal torsion}.
Finally, by the asymptotics \ref{pro:anal torsion and its trunc}
this finishes the proof of the theorem.

\subsection{Proof of Theorem \ref{thm:analyt torsion for ample}}

By the Kodaira vanishing theorem we have that $H^{0,q}(X,kL)=\{0\}$
for $k>>1$ (since $L$ is assumed ample) and hence $\mathcal{L}^{0,q}(k\phi_{t})=0$
for $q>0.$ By Theorem A in \cite{b-b} $(\mathcal{L}(\phi_{1})-\mathcal{L}_{k}(\phi_{0})/k^{(n+1)}$
converges to $(\mathcal{E}(P\phi_{1})-\mathcal{E}(P\phi_{0})$ and
hence the previous theorem (or rather its proof) gives 
\[
\lim_{k\rightarrow\infty}\log\left(T_{]0,\lambda_{k}[}(e^{-k\phi_{1}},\omega_{0})/T_{]0,\lambda_{k}[}(e^{-k\phi_{0}},\omega_{0})\right)=(\mathcal{E}(\phi_{1})-\mathcal{E}(\phi_{0}))-(\mathcal{E}(P\phi_{1})-\mathcal{E}(P\phi_{0}))
\]
Since, $\phi_{0}$ is assumed positively curved we have on one hand
that $P\phi_{0}=\phi_{0}$ so that the terms involving $\phi_{1}$
in the rhs above cancel. On the other hand, the smallest positive
eigenvalue $\lambda_{k}^{0,q}$of $\Delta_{k\phi_{0}}^{0,q}$ satisfies
$\lambda_{k}^{0,q}\geq Ck$ (see below) and hence the second term
in the lhs above also vanishes. All in all this gives the convergence
of the truncated analytic torsion in Theorem \ref{thm:analyt torsion for ample}
and by Proposition  \ref{pro:anal torsion and its trunc} the same
asymptotics then hold for the analytic torsions. As for the eigenvalue
estimate used above it is a standard consequence of the Kodaira-Nakano
identity \cite{gr-ha} in the case when $\omega$ is Kähler. In the
non-Kähler case it was shown in \cite{b-vas} (Theorem 1).
\begin{rem}
The eigenvalue estimate referred to above can also be deduced from
the Kähler case as follows: first one notes that the first positive
eigenvalue on the subspace $\ker\bar{\partial}$ of the $\bar{\partial}-$Laplacian
associated to $(k\phi,h_{X})$ is the supremum over all constants
$C_{k}^{(q)}(h_{X})$ such that for any $f,$ a $\bar{\partial}-$closed
$(0,1)-$form with values in $L,$ the inhomogeneous $\bar{\partial}-$equation
$\bar{\partial}u=f$ for can be solved with an estimate 
\[
\left\Vert u\right\Vert _{(k\phi,h_{X})}^{2}\leq\frac{1}{C_{k}^{(q)}(\omega)}\left\Vert f\right\Vert _{(k\phi,h_{X})}^{2}.
\]
But since $X$ is compact there is a constant $A$ such that $A^{-1}\omega_{0}\leq h_{X}\leq A\omega_{0}$
for $\omega_{0}$ a fixed Kähler metric and hence replacing $h_{X}$
with $\omega_{0}$ only distorts $C_{k}^{(q)}(h_{X})$ with a multiplicative
constant independent of $k.$ Finally, using the bijection \ref{eq:of of deriv of log det: bijec}
(and since $H^{0,q}(X,kL)=\{0\}$ for $k>>1)$ and the fact that $\bar{\partial}$
commutes with $\Delta$ this gives the desired lower bound for the
first positive eigenvalue on all of $\Omega^{0,q}(X,kL)$ for all
$q.$
\end{rem}

\subsection{Proof of Corollary \ref{cor:exist of small eig}}

Let $\lambda_{k}=k^{1-\epsilon}$ for some $\epsilon\in]1/2,1[.$By
Proposition  \ref{pro:strong morse} the number of terms in the sum
\begin{equation}
k^{-(n+1)}\log\det(1_{]0,\lambda_{k}[}(\Delta_{\bar{\partial}}^{(q)})=\frac{1}{k^{n}}\sum_{i}\frac{1}{k}\log\lambda_{i,k}^{(q)}\label{eq:sum of log einvalues for q}
\end{equation}
grows as constant times $k^{n}.$ Hence, if there were no sequence
of eigenvalues as in the statement of the corollary then the previous
sum would converge to zero for any $q$ contradicting the positivity
assumption in the corollary (according to the previous theorem). As
for the final statement we first note that, when $n=1:$ 
\[
\mathcal{E}(P\phi,\phi)=\frac{1}{2}\left\Vert d(P\phi-\phi)\right\Vert _{X}^{2},
\]
which follows from integration by parts, using the ``orthogonality
relation'' $\int(P\phi-\phi)dd^{c}(P\phi)=0$ \cite{b-b}. Hence
$\mathcal{E}(P\phi,\phi)\geq0$ with equality iff $P\phi=\phi+c$
for some constant $c.$ But then it follows from the definition of
$P\phi$ that $c=0,$ i.e. $\phi$ has semi-positive curvature. 
\begin{rem}
\label{rem:referee}By way of comparison let us briefly explain how
to deduce the asymptotics for the (non-truncated) Quillen metric Theorem
\ref{thm:quillen} from the results in \cite{b-g-s,b-v-,m-m} (following
the suggestions of a referee). First, by \cite[III, Theorem 1.18]{b-g-s}
(or \cite[Theorem 5.5.6]{m-m}) the differential of the Quillen metric
associated to $(kL,e^{-k\phi},h_{X})$ can be expressed as the constant
term $C_{k,0}$ in the $t-$expansion $C_{k,t}$ of the super trace
of the heat-kernel associated to the Kodaira-Dolbeault Laplacian (this
is the ``non-truncated'' version of Lemma \ref{lem:deriv of trunc analyt torsion}).
By \cite[Theorem 2]{b-v-} the corresponding heat-kernel asymptotics
are uniform in $t\in[0,T]$ as $k\rightarrow\infty$ and calculating
$C_{k,0},$ by comparing with a constant coefficient model operator,
gives that $C_{k,0}=k^{n+1}(dd^{c}\phi)^{n}+o(k^{n+1})$ (this is
the ``non-truncated'' heat-kernel analog of Proposition \ref{pro:strong morse}).
Finally, the proof is concluded by integrating along line segments
in the space of metrics on $L,$ just as above (see also \cite[Theorem 5.5.9]{m-m}
for a detailed proof in the non-Kähler case).
\end{rem}

\section{Further remarks on tunneling and outlook}

\subsection{\label{sub:Comparison-with-the}Comparison with the tunneling effect
for Witten Laplacians and large deviations for fermions}

It may be illuminating to compare the asymptotics of the truncated
analytic torsions in Theorem \ref{thm:analyt torsion for ample} with
the well-known results concerning the tunneling effect for Witten's
deformation of the De Rham complex, which appeared in Witten's heuristic
approach to Morse theory based on supersymmetric quantum mechanics
\cite{wi} (see \cite{h-s} for rigorous results based on semi-classical
analysis). Geometrically this latter setting corresponds to letting
$L\rightarrow X$ be the\emph{ trivial} line bundle over a\emph{ real}
manifold $X$ and replacing the operator $\bar{\partial}$ with the
exterior derivative $d.$ Any given \emph{global} function $\phi$
on $X$ induces a Hermitian metric $h_{kL}$ on $kL$ represented
as $h_{kL}=e^{-k\phi}$ in the standard global trivialization (i.e.
$s=1)$ of $L,$ where now $k$ makes sense for any positive number
and where the semi-classical parameter $\hbar:=1/k$ corresponds to
Planck's constant in the quantum mechanical picture. To this setting
one associates as before a Laplacian on $(0,q)-$forms, $\Delta_{k}^{(q)},$
depending on $k\phi,$ which may be identified with the\emph{ Witten
Laplacian} (in a unitary frame). Assuming that $\phi$ is a Morse
function we denote by $X(q)$ the finite set of all critical points
of $\phi$ where the Hessian of $\phi$ has index $q.$ As shown in
\cite{h-s} the set of all ``small'' eigenvalues $(\lambda_{i,k}^{(q)})$
of the Witten Laplacian $\Delta_{k}^{(q)}$ is in a one to one correspondence
with the finite set $X(q).$ Moreover, all the non-zero eigenvalues
are actually \emph{exponentially} small: 
\begin{equation}
\lim_{k\rightarrow\infty}\frac{1}{k}\log\lambda_{i,q}^{(k)}=-c_{i,q}(\phi),\label{eq:witten eigenv asympt}
\end{equation}
where the positive number $c_{i,q}(\phi)$ may be expressed in terms
of an Agmond distance betweeen critical points. In the present complex
geometric setting of the $\bar{\partial}-$Laplacian the number of
``small'' eigenvalues $N_{q}^{(k)}$ depends on $k$ and is of the
order $O(k^{n}),$ more precisely: 
\[
N_{q}^{(k)}=k^{n}\int_{X(q)}(-1)^{q}(dd^{c}\phi)^{n}/n!+o(k^{n}),
\]
where now $X(q)$ is the subset of $X$ where the curvature form $dd^{c}\phi$
(i.e. the \emph{complex} Hessian of $\phi)$ has index $q$ (as shown
by Demailly in his proof of his holomorphic Morse inequalities \cite{d1};
compare the discussion in the beginning of Section \ref{sub:Small-vs.-exponentially}
below). Hence the asymptotics of the truncated analytic torsions in
Theorem \ref{thm:analyt torsion for ample} may be interpreted as
an averaged version of \ref{eq:witten eigenv asympt}. This analogy
becomes particularly striking in the Riemann surface case (compare
the asymptotics \ref{eq:asym in them exp small}).

The approach in \cite{h-s} to study the asymptotics of the exponentially
small eigenvalues, i.e. of the quantum mechanical\emph{ tunneling
effect, }uses that the Witten Laplacians are\emph{ elliptic} when
viewed as semi-classical elliptic differential operators. In the particular
case when $q=0$ there is also a probabilistic approach to the corresponding
Witten Laplacian $\Delta_{k}^{(0)}$ which refines the information
in \ref{eq:witten eigenv asympt} with asymptotics for pre-factors
under certain genericity assumptions (see \cite{bgk} and references
therein for the case of Euclidean domains). The point is that the
operator $\Delta_{k}^{(0)}$ is the Feller semigroup generator for
the stochastic differential equation on the Riemannian manifold $X$
describing a Brownian particle in the gradient field of $-\phi$ at
temperature $1/k.$ The invariant probability measure of the process
is 
\begin{equation}
\mu_{k\phi}:=\frac{e^{-k\phi}dV}{Z_{k\phi}}\label{eq:invariant meas}
\end{equation}
and its relation to the exponentially small eigenvalues was studied
in \cite{hks} in connection to simulated annealing. As explained
in \cite{bgk} each local minima of $\phi$ corresponds to a metastable
state and the corresponding eigenvalue to the \emph{inverse life time}
of the state (i.e. the inverse mean exit time). See also the very
recent work \cite{p-n-v} for refinements of \ref{eq:witten eigenv asympt}
for general degrees $q.$

There are also quantum mechanical, as well as probabilistic motivations
for considering the present complex geometric setting. For example
in the Riemann surface case the $\bar{\partial}-$Laplacian is the
Pauli Hamiltonian for a single fermion with spin up ($q=0)$ and spin
down $(q=1)$ subject to the magnet field with two-form $F_{kA}:=kdd^{c}\phi.$
Moreover, from a probabilistic point of view the asymptotics of the
corresponding truncated analytic torsions appear in the large deviation
principle for the corresponding determinantal random point process
on $X$ with $N_{k}$ particles defined by the following probability
measure on $X^{N_{k}}:$ 
\[
\mu_{k\phi}^{(N_{k})}:=\frac{|\det\Psi|^{2}(x_{1},....x_{N_{k}})e^{-k(\phi_{1}(x_{1})+\cdots+\phi_{1}(x_{N}))}dV^{\otimes N_{k}}}{\mathcal{Z}_{k\phi}},
\]
where $N_{k}$ is the dimension of the space $H^{0}(X,kL)$ of global
holomorphic sections of $kL$ and $\det\Psi$ is any generator of
the complex line $\det H^{0}(X,kL),$ which we identify with a holomorphic
section over $X^{N_{k}}$ \cite{berm2}. Physically, this is the maximally
filled fermionic many-particle ground state and $\det\Psi$ is the
corresponding Slater determinant. The large deviation principle referred
to above may be symbolically formulated as follows: 
\begin{equation}
\mbox{Prob }\left(\frac{1}{N_{k}}\sum_{i}\delta_{x_{i}}\in\mathcal{B}_{\epsilon}(\mu)\right)\sim\prod_{q=1}^{n}(\det\Delta_{\bar{\partial},k}^{0,q})^{-(-1)^{q+1}q}e^{-kN_{k}E(\mu)},\label{eq:ldp}
\end{equation}
 where, by Theorem \ref{thm:analyt torsion for ample} above, the
weighted determinant above may by replaced the weighted product of
the corresponding exponentially small eigenvalues of the Kodaira-Dolbeault
Laplacian. In the left hand side above $\mathcal{B}_{\epsilon}(\mu)$
denotes a ball of radius $\epsilon$ centered at a given probability
measure $\mu$ in the space of all probability measures on $X$ (equipped
with a fixed metric) and $E(\mu)$ appearing in the right hand side
denotes the pluricomplex energy of the measure $\mu,$ defined with
respect to the curvature form $\omega.$ The relation \ref{eq:ldp}
holds asymptotically as $N\rightarrow\infty$ and $\epsilon\rightarrow0$
and (see \cite{berm2} for the precise statement). Loosely speaking
the left hand side in formula \ref{eq:ldp} gives the probability
of finding a ``cloud'' of particles at the points $x_{1},x_{2},...,x_{N}$
in $X$ representing a macroscopic state $\mu,$ in the sense that
$\frac{1}{N_{k}}\sum_{i}\delta_{x_{i}}$ approximates $\mu$ (the
small parameter $\epsilon$ can physically be seen as a measure of
the coarse-graining). As explained in \cite{berm2} the formula \ref{eq:ldp}
may also be interpreted as an (asymptotic) higher dimensional bosonization
formula, generalizing the well-known exact bosonization formula \cite{b-v-}
on a Riemann surface.

Note that in the setting of the Witten Laplacian the role of $H^{0}(X,kL)$
is played by the \emph{one-}dimensional kernel of $d$ acting on the
space smooth functions, inducing the \emph{one} particle point-process
\ref{eq:invariant meas}. The main analytical difference here is that
the $\bar{\partial}-$Laplacian on $(0,q)-$forms with values in $kL$
is \emph{not }an elliptic operator, when viewed as a semi-classical
differential operator (see \cite{b-sj}). There appears to be very
few results concerning the tunneling effect for non-elliptic semi-classical
operators (see however \cite{h-h-s} and references therein for the
study of Kramers Fokker-Planck type operators generalizing the Witten
Laplacian for $q=0$).

\subsection{\label{sub:Small-vs.-exponentially}``Small'' vs. ``exponentially
small'' eigenvalues for the Dolbeault complex }

Fix a number $\epsilon\in]0,1/2[.$ For a given positive integer $k$
we say that an eigenvalue $\nu_{k}$ of the $\bar{\partial}-$Laplacian
$\Delta_{\bar{\partial}}$ associated to $kL$ is ``small'' if $\nu_{k}\leq k^{1-\epsilon}.$
By Proposition  \ref{pro:strong morse} the number of ``small''
eigenvalues is of the order $O(k^{n}).$ Hence, as explained above
the ``small'' eigenvalues which are not ``exponentially small''
make no contribution to the large $k$ limit of the sum \ref{eq:sum of log einvalues for q}.
But it is tempting to ask whether \emph{all }``small'' eigenvalues
are actually ``exponentially small''? As recalled in section \ref{sub:Comparison-with-the}
this is indeed the case in the setting of the Witten Laplacian under
the assumption that $\phi$ be a Morse function. However in the present
setting there are no non-degeneracy assumptions on the weight $\phi$
of the metric and hence it seems unlikely that the answer to the question
above is yes, in general. Still when $q=0$ the number of ``small''
and ``exponentially'' small eigenvalues are the same to the leading
order according to the following proposition (that should be compared
with Theorem 3.14 in \cite{d1} concerning eigenvalues of the order
$k\lambda$ for $\lambda$ fixed):
\begin{prop}
\label{pro:quasi-mode}Given any $\delta>0$ there is a constant $C_{\delta}$
such that the number $N_{\delta,k}$ of eigenvalues in $[0,e^{-C_{\delta}K}]$
of the $\bar{\partial}-$Laplacian $\Delta_{\bar{\partial}}$ acting
on smooth sections of $kL$ satisfies 
\[
\int_{X(0)}(dd^{c}\phi)^{n}/n!-\delta\leq N_{\delta,k}\leq N_{k}\leq\int_{X(0)}(dd^{c}\phi)^{n}/n!+\delta
\]
 for $k$ sufficiently large, where $N_{k}$ is the number of ``small''
eigenvalues. \end{prop}
\begin{proof}
Given $\delta>0$ take $K_{\delta}$ a compact subset of the open
set $X(0):=\{dd^{c}\phi>0\}$ such that 
\[
\int_{X(0)}(dd^{c}\phi)^{n}/n!-\delta\leq\int_{K_{\delta}}(dd^{c}\phi)^{n}/n!.
\]
By Fatou's lemma it will be enough to prove that 
\[
\liminf_{k\rightarrow\infty}k^{-n}\B_{[0,e^{-C_{\delta}K}[}^{0,q}(x)\geq(dd^{c}\phi)^{n}(x)/n!
\]
for any given $x\in K_{\delta}.$ As shown in \cite{berm0} (Proposition
 5.3) it is enough to, given a point $x\in K_{\delta},$ find a smooth
section $\alpha_{k}$ of $kL$ such that 
\[
(i)\,\liminf_{k\rightarrow\infty}k^{-n}\frac{(|\alpha_{k}|^{2}e^{-k\phi})(x)}{\int_{X}|\alpha_{k}|^{2}e^{-k\phi}dV}\geq(dd^{c}\phi)^{n}(x)/n!,\,\,\,\,(ii)\, k^{-1}\left\langle \Delta_{\bar{\partial}}\alpha_{k},\alpha_{k}\right\rangle _{k\phi}\leq\delta_{k}
\]
 for $\delta_{k}=e^{-2C_{\delta}k}.$ To this end we simply take $\alpha_{k}=\chi s$
where $\chi$ is a smooth cut-off function which is equal to $1$
close to $x$ and $s$ is local holomorphic frame for $L$ on a neighborhood
$x$ such that the corresponding local weight $\phi(z)$ is given
by $\phi(z)=\sum_{i=1}^{n}\mu_{i}|z_{i}|^{2}+O(|z|^{3})$ (where $\mu_{i}>0$
since $dd^{c}\phi>0$ at $x).$ Then $(i)$ above follows from computing
a Gaussian integral. Moreover, 
\[
\left\langle \Delta_{\bar{\partial}}\alpha_{k},\alpha_{k}\right\rangle _{k\phi}=\left\langle \bar{\partial}\alpha_{k},\bar{\partial}\alpha_{k}\right\rangle _{k\phi}=\int|(\partial\chi)|^{2}e^{-k\phi}dV\leq e^{-2C_{\delta}k}.
\]
By the compactness of $K_{\delta}$ the constant $C_{\delta}$ can
be taken to be independent of $x.$ This finishes the proof of the
lower bound on $N_{\delta,k}$ and the upper bound is a special case
of Proposition  \ref{pro:strong morse}.
\end{proof}
It would be interesting to know if the analog for $q>0$ of the previous
proposition is also valid? If one replaces $e^{-kC_{\delta}}$ with
a sequence $\delta_{k}$ of the form $\delta_{k}=O(k^{-\infty}),$
i.e. $\delta_{k}\leq C_{m}k^{-m}$ for any $m>0$ then this is indeed
case. This follows from taking $\alpha_{k}$ in the previous proof
to be the given by $\Pi_{k}(x,\cdot)$ where $\Pi_{k}(x,y)$ is the
local projector for $(0,q)-$forms constructed in \cite{b-sj}, defining
a local Fourier integral operator with a complex phase (in fact, only
the phase function constructed in \cite{b-sj} is needed together
with the leading term in the symbol expansion).

\subsection{\label{sub:Discussion-and-outlook}Discussion and outlook}

This paper gives a first small step towards understanding the tunneling
effect for the asymptotic Dolbeault complex associated to high powers
of a line bundle $L\rightarrow X$ equipped with a (non-positively
curved) Hermitian metric. Ideally, one could hope to give an asymptotic
description, as $k\rightarrow\infty,$ of the following measures on
$\R$
\[
\nu_{q}^{(k)}:=-\frac{1}{k^{n+1}}\sum_{i=1}^{N_{q}^{(k)}}\log\lambda_{i,q}^{(k)},
\]
where $\{\lambda_{i,q}^{(k)}\}_{i=1,}^{N_{q}^{(k)}}$ ranges over
the ``small'' non-zero eigenvalues of the Kodaira-Dolbeault Laplacian
acting on $\Omega^{0,q}(X,L^{\otimes k}).$ Conjecturally, $\nu_{q}^{(k)}$
converges weakly as $k\rightarrow\infty$ to a measure $\nu_{q}$
only depending on the fixed metric on $L.$ If such a measure exists,
say when $n=1,$ Corollary \ref{cor:exist of small eig} shows that
its first moment is given by a Dirichlet norm: 
\[
\int_{\R}td\nu_{1}(t)=\frac{1}{2}\left\Vert d(P\phi-\phi)\right\Vert _{X}^{2}.
\]
 Furthermore, it would be very interesting to give an asymptotic description
of the corresponding eigenforms. In the De Rham case it is well-known
that the eigenforms are, for a generic Morse function, localized along
the gradient paths of the Morse function connecting critical points
of index $q$ and index $q+1$ \cite{h-s} (in physics terminology
these are the ``tunneling paths'' or ``instanton paths''\cite{wi}).
It seems likely that, in the asymptotic Dolbeault case, the role of
the gradient paths (at least for $q=0)$ are played by the leaves
of the Monge-Ampère foliation attached to the weight $\psi:=P\phi$
on the open subset $\Omega$ of $X$ where $P\phi<\phi.$ Recall that
the corresponding leaves are defined as the holomorphic curves along
which $\psi$ is locally harmonic. Since, $(dd^{c}\psi)^{n}=0$ such
leaves are well-known to exist under appropriate (but rather strong)
regularity assumptions, but in the present case when $\psi$ is merely
known to be $\mathcal{C}^{1,1}-$smooth the existence of the leaves
is a very delicate issue. Still, one could hope to be work with a
suitable approximate notion of leaves.

\end{document}